\documentclass[12pt]{amsart}
\usepackage[a4paper, total={5.35 in, 8.3 in}]{geometry}
\usepackage{amsmath,amsfonts,amssymb}
\usepackage{amsfonts}
\usepackage{verbatim}
\usepackage{enumerate}
\usepackage{amsthm}
\usepackage{url}
\usepackage[colorlinks]{hyperref}
\usepackage[english]{babel}
\usepackage{tikz}

\newcommand{\Z}{\mathbb Z}

\newcommand{\fqn}{\mathbb{F}_{q^n}}
\newcommand{\F}{\mathbb{F}}

\newcommand{\ord}{\mathrm{ord}}

\newcommand{\q}{\mathcal P}
\newcommand{\m}{\mu_q}
\newcommand{\R}{\rho_q}

\newtheorem{theorem}{Theorem}[section]

\newtheorem{proposition}[theorem]{Proposition}
\newtheorem{definition}[theorem]{Definition}

\newtheorem{lemma}[theorem]{Lemma}
\newtheorem{corollary}[theorem]{Corollary}

\author[Lucas Reis]{Lucas Reis}
\address{Departamento de Matem\'{a}tica, Universidade Federal de Minas Gerais, Belo Horizonte MG, 31270901, Brazil}
\email{lucasreismat@gmail.com}

\title{The average density of K-normal elements over finite fields}

\date{\today
}


\begin{document}
\maketitle

\begin{abstract}
Let $q$ be a prime power and, for each positive integer $n\ge 1$, let $\F_{q^n}$ be the finite field with $q^n$ elements. Motivated by the well known concept of normal elements over finite fields, Huczynska et al (2013) introduced the notion of $k$-normal elements. More precisely, for a given $0\le k\le n$, an element $\alpha\in \F_{q^n}$ is $k$-normal over $\F_q$ if the $\F_q$-vector space generated by the elements in the set $\{\alpha, \alpha^q, \ldots, \alpha^{q^{n-1}}\}$ has dimension $n-k$. The case $k=0$ recovers the normal elements. If $q$ and $k$ are fixed, one may consider the number $\lambda_{q, n, k}$ of elements $\alpha\in \F_{q^n}$ that are $k$-normal over $\F_q$ and the density $\lambda_{q, k}(n)=\frac{\lambda_{q, n, k}}{q^n}$ of such elements in $\F_{q^n}$. In this paper we prove that the arithmetic function $\lambda_{q, k}(n)$ has positive mean value, in the sense that the limit
 $$\lim\limits_{t\to +\infty}\frac{1}{t}\sum_{1\le n\le t}\lambda_{q, k}(n),$$
exists and it is positive. 
\end{abstract}

\section{Introduction}
Let $q$ be a prime power, let $n$ be a positive integer and let $\F_{q^n}$ be the finite field with $q^n$ elements. The field $\F_{q^n}$ can be viewed as an $\F_q$-vector space of dimension $n$. In this context, an element $\beta\in \F_{q^n}$ is {\em normal} over $\F_q$ if the $\F_q$-vector space generated by the set $\{\beta, \ldots, \beta^{q^{n-1}}\}$ has dimension $n$, i.e., this set is an $\F_q$-basis for $\F_{q^n}$. 

Normal elements are quite useful in applications such as computer algebra, due to their efficiency on basic operations (most notably, the exponentiation). We refer to~\cite{GAO}  (and the references therein) for an overview on normal elements, including theoretical and practical aspects.

Motivated by the normal elements over finite fields, in~\cite{HMPT} the authors introduced the concept of $k$-normal elements. More precisely, for $0\le k\le n$, an element $\alpha\in \F_{q^n}$ is $k$-normal over $\F_q$ if the $\F_q$-vector space generated by the set $\{\alpha, \ldots, \alpha^{q^{n-1}}\}$ has dimension $n-k$. In particular, normal elements are just $0$-normal elements. In the same paper, the authors obtained results on the existence and number of $k$-normal elements and proposed many problems. Since then, many papers have provided results on $k$-normal elements, including $k$-normal elements with prescribed multiplicative order: see~\cite{jv, ag, MA, KR, RT18, R19, T} for more details.

Back to normal elements, some works have studied the proportion of elements in $\F_{q^n}$ that are normal over $\F_q$: see~\cite{F00, GP}. More recently, in~\cite{R}, we proved that the proportion $\lambda_{q, 0}(n)$ of elements in $\F_{q^n}$ that are normal over $\F_q$ is an arithmetic function with positive mean value, in the sense that the limit
$$\lim\limits_{t\to +\infty}\frac{1}{t}\sum_{1\le n\le t}\lambda_{q, 0}(n),$$
exists and it is positive. Motivated by the latter, in this paper we provide a generalization of this result to $k$-normal elements. More precisely, we have the following theorem.

\begin{theorem}\label{thm:main}

Let $q$ be a prime power and, for integers $n\ge 1$ and $0\le k\le n-1$, let $\lambda_{q, n, k}$ be the number of elements $\alpha\in \F_{q^n}$ such that $\alpha$ is $k$-normal over $\F_q$. If we set $\lambda_{q, k}(n)=\frac{\lambda_{q, n, k}}{q^n}$, then there exists $\overline{\lambda}_{q, k}\ge 0$ such that
$$\overline{\lambda}_{q, k}=\lim\limits_{t\to +\infty}\frac{1}{t}\sum_{1\le n\le t}\lambda_{q, k}(n).$$
\end{theorem}
Theorem~\ref{thm:main} entails that there is a mean value for the function $\lambda_{q, k}(n)$. It turns out that, for a positive proportion of natural numbers $n$, we have that $\lambda_{q, k}(n)\ge \varepsilon\lambda_{q, 0}(n')$, where $n'$ is a liner function on $n$ and $\varepsilon$ depends only on $k$. From this fact and some results from~\cite{R}, we are able to obtain the following result.

\begin{corollary}\label{cor:main}
If $\overline{\lambda}_{q, k}$ is as in Theorem~\ref{thm:main}, we have that $$\overline{\lambda}_{q, k}\ge \frac{\overline{\lambda}_{q, 0}}{p^tq^k}>0,$$ where $p$ is the characteristic of $\F_q$, $t=\lfloor\log_pk\rfloor+1$ if $k>0$ and $t=0$ if $k=0$. Moreover, $\overline{\lambda}_{q, 0}>1-\frac{1}{\sqrt{q}}-\frac{1}{q}$ if $q\ge 4$. 
\end{corollary}

The proof of Theorem~\ref{thm:main} relies on standard ideas from Analytic Number Theory, combined with some estimates on sums involving arithmetic functions over the polynomial ring $\F_q[X]$.

The paper is structured as follows. In Section 2 we provide all the basic machinery that is further used in Section 3, were we prove Theorem~\ref{thm:main} and Corollary~\ref{cor:main}.

\section{Preparation}
In this section we provide some background material that is further used throughout the paper. We start with some basics on arithmetic functions over $\Z$ and $\F_q[X]$.

\begin{definition}
For positive integers $a, b$ and a monic polynomial $F\in \F_q[X]$, we define the following arithmetic functions:

\begin{enumerate}[(i)]
\item $\varphi(a)$ is the number of integers $1\le i\le a$ such that $\gcd(a, i)=1$;
\item if $\gcd(a, b)=1$, set $\ord_ba=\min\{j>0\,|\, a^j\equiv 1\pmod b\}$; 
\item set $\mu_q(1)=1$,  $\mu_q(F)=0$ if $F$ is not squarefree and $\mu_q(F)=(-1)^r$ if $F$ is the product of $r$ distinct irreducible polynomials over $\F_q$;
\item $\Phi_q(F)$ is the number of invertible elements in the quotient ring $\frac{\F_q[X]}{F(X)\F_q[X]}$.
\end{enumerate}
\end{definition}

The following result is directly verified.

\begin{lemma}\label{lem:phi}
For a monic polynomial $F\in \F_q[X]$, we have that 
$$\frac{\Phi_q(F)}{q^{\deg(F)}}=\prod_{H|F}\left(1-\frac{1}{q^{\deg(H)}}\right),$$
where $H$ runs  over the distinct monic divisors of $F$ that are irreducible over $\F_q$. In particular, 
$$\frac{\Phi_q(F)}{q^{\deg(F)}}=\sum_{G|F}\frac{\mu_q(G)}{q^{\deg(G)}},$$
where $G$ runs over the monic divisors of $F$, defined over $\F_q$.
\end{lemma}

For a polynomial $F\in \F_q[X]$ that is not divisible by $X$, it is clear that there exists a positive integer $j$ such that $F(X)|X^j-1$. We set $\ord(F)=\min\{j>0\,|\, F(X)|X^j-1\}$, the order of $F$. We have the following result.

\begin{lemma}\label{lem:ord}
Let $F, G\in \F_q[X]$ be polynomials that are not divisible by $X$. Then the following hold:

\begin{enumerate}[(i)]
\item $\ord(FG)\ge \max\{\ord(F), \ord(G)\}$;
\item if $n$ is a positive integer, then $F(X)|X^n-1$ if and only if $n$ is divisible by $\ord(F)$;  
\item if $F$ is irreducible, $E=\ord(F)$ and $i=\deg(F)$, then $\ord_Eq=i$;
\item conversely, for each positive integer $E$ such that $\ord_Eq=i$, there exist $\frac{\varphi(E)}{i}$ monic irreducible polynomials $F\in \F_q[X]$ such that $\ord(F)=E$ and $\deg(F)=i$. 
\end{enumerate}
\end{lemma}

\begin{proof}
    Items (i) and (ii) follow directly by the definition of $\ord(F)$. Items (iii) and (iv) follow by Theorem 2.47 in~\cite{LiNi} and the fact that  any monic irreducible polynomial $F$ with $\ord(F)=E$ is necessarily an irreducible factor (over $\F_q$) of the $E$-th cyclotomic polynomial.

\end{proof}

The following lemma provides a formula for the number of elements in $\F_{q^n}$ that are $k$-normal over $\F_q$. Its proof follows directly by  Theorem 3.5 in~\cite{HMPT}.

\begin{lemma}\label{eq:count} The number $\alpha_{q, n, k}$   of elements in $\F_{q^n}$ that are $k$-normal over $\F_q$ is given by
\begin{equation*}\label{eq:num-k-normal}\alpha_{q, n, k}=\sum_{F|X^n-1\atop{\deg(F)=k}}\Phi_q\left(\frac{X^n-1}{F}\right),\end{equation*}
where the divisors are monic and polynomial division is over $\F_q$.
\end{lemma}

We end this section with some useful inequalities. From the main result in~\cite{NR83}, we have the following estimate.

\begin{lemma}\label{lem:estimate-divisor} If $\sigma_0(m)$ is the number of positive divisors of $m$, then for every $m\ge 3$,
$$\sigma_0(m)<m^{\frac{1.1}{\log\log m}}.$$
\end{lemma}

The following lemma provide some basic inequalities that can be directly verified. We omit details.

\begin{lemma}\label{lem:ineq}
For every positive integer $N$, the following hold:

\begin{enumerate}[(i)]
\item $\left(1+\frac{1}{N}\right)^N<e$, where $e$ is the Euler number;
\item $\sum_{j=1}^N\frac{1}{j}\le \log j+1$;
\item the number of monic irreducible polynomials of degree $N$ over $\F_q$ is at most $\frac{q^N}{N}$.
\item $\sum_{i|N}\varphi(i)=N$. 
\end{enumerate}

\end{lemma}

\section{Proof of Theorem~\ref{thm:main}}
In what follows, for real valued functions $\mathcal F, \mathcal G$, we write $\mathcal F(t)=O(\mathcal G(t))$ if $|\mathcal F(t)|\le C\cdot |\mathcal G(t)|$ for some absolute constant $C>0$ and write $\mathcal F(t)=o(\mathcal G(t))$ if $\lim\limits_{t\to +\infty}\frac{\mathcal F(t)}{\mathcal G(t)}=0$.

Let $\mathcal M_k$ denote the set of monic polynomials $F\in \F_q[X]$ such that $F(X)$ is not divisible by $X$ and set $\mathcal M=\bigcup_{k\ge 0}\mathcal M_k$. For each $F\in \mathcal M_k$ and each real number $t>0$, set
$$S_F(t)=\sum_{1\le n\le t\atop F|X^n-1}\frac{\Phi_q\left(\frac{X^n-1}{F}\right)}{q^{n-k}}.$$

From Lemma~\ref{eq:count}, we have that 

$$\frac{1}{t}\sum_{1\le n\le t}\lambda_{q, k}(n)=\frac{1}{tq^k}\sum_{F\in \mathcal M_k}S_F(t).$$
Since the set $\mathcal M_k$ is finite it suffices to prove that, for each $F\in \mathcal M_k$, the limit $\lim\limits_{t\to +\infty}\frac{1}{t}\sum_{F\in \mathcal M_k}S_F(t)$ exists. From Lemma~\ref{lem:phi}, we have that

\begin{equation}\label{eq:1}S_F(t)=\sum_{1\le n\le t\atop F|X^n-1}\sum_{G|\frac{X^n-1}{F}}\frac{\mu_q(G)}{q^{\deg(G)}}=\sum_{1\le n\le t}\sum_{FG|X^n-1}\frac{\mu_q(G)}{q^{\deg(G)}}.\end{equation}
For each $G\in \mathcal M$, set $a_G=\ord(FG)$. From Lemma~\ref{lem:ord}, the term $\frac{\mu_q(G)}{q^{\deg(G)}}$ contributes only for the integers $n$ that are divisible by $a_G$. From this fact,  Eq.~\eqref{eq:1} implies that

\begin{equation}\label{eq:2}S_F(t)= \sum_{G\in \mathcal M\atop 1\le a_G\le t}\left\lfloor\frac{t}{a_G}\right\rfloor\frac{\mu_q(G)}{q^{\deg(G)}}=t\cdot M_F(t)+R_F(t),
\end{equation}
where $M_F(t)=\sum_{G\in \mathcal M\atop 1\le a_G\le t}\frac{\mu_q(G)}{a_Gq^{\deg(G)}}$ and $R_F(t)=\sum_{G\in \mathcal M\atop 1\le a_G\le t}\left\{\frac{t}{a_G}\right\}\frac{\mu_q(G)}{q^{\deg(G)}}$.
From Lemma~\ref{lem:ord}, we have that $a_G\ge \ord(G)$. Therefore, we obtain that $$|M_F(t)|\le \sum_{G\in \mathcal M^*\atop 1\le \ord(G)\le t}\frac{1}{\ord(G)q^{\deg(G)}}=:M_F^*(t),$$ and $$|R_F(t)|\le \sum_{G\in \mathcal M^*\atop 1\le \ord(G)\le t}\frac{1}{q^{\deg(G)}}=:R_F^*(t),$$ where $\mathcal M^*$ is the set of squarefree elements of $\mathcal M$. Since each term in the sums defining $M_F^*$ and $R_F^*$ are positive, Theorem~\ref{thm:main} follows by Eq.~\eqref{eq:2} and the following result.

\begin{proposition}
The following hold:

\begin{enumerate}[(i)]
    \item the sum $\sum_{G\in \mathcal M^*}\frac{1}{\ord(G)q^{\deg(G)}}$ is finite;

    \item $R_F^*(t)=o(t)$.
\end{enumerate}
\end{proposition}
\begin{proof}
    We prove the items separately.

    \begin{enumerate}[(i)]
        \item For each integer $i\ge 1$, let $\mathcal I_i$ be the set of monic  polynomials of degree $i$ that are irreducible over $\F_q$. Also, for each integer $i\ge 1$ and each $f\in \mathcal I_i$, let $\mathcal M(i, f)$ be the set of polynomials $G\in \mathcal M^*$ such that $f$ divides $G$ and every irreducible divisor of $G$ has degree at most $i$. Hence we have that 

        $$\sum_{G\in \mathcal M^*}\frac{1}{\ord(G)q^{\deg(G)}}\le 1+\sum_{i\ge 1}\sum_{f\in \mathcal I_i}\sum_{G\in \mathcal M(i, f)} \frac{1}{\ord(G)q^{\deg(G)}}.$$
        If $G\in \mathcal M(i, f)$, then $f$ divides $G$. From Lemma~\ref{lem:ord}, we obtain that $\ord(G)\ge \ord(f)$ and so 
$$\sum_{i\ge 1}\sum_{f\in \mathcal I_i}\sum_{G\in \mathcal M(i, f)} \frac{1}{\ord(G)q^{\deg(G)}}\le \sum_{i\ge 1}\sum_{f\in \mathcal I_i}\frac{1}{\ord(f)}\sum_{G\in \mathcal M(i, f)} \frac{1}{q^{\deg(G)}}.$$
We observe that each element of $\mathcal M(i, f)$ is written uniquely as $\prod_{g\in C_i}g^{e_g}$, where $C_i=\bigcup_{1\le j\le i}\mathcal I_j$, $e_g\in \{0, 1\}$ and $e_f=1$. Hence 
$$\sum_{G\in \mathcal M(i, f)} \frac{1}{q^{\deg(G)}}\le \frac{1}{q^i}\prod_{j=1}^i\left(1+\frac{1}{q^j}\right)^{|\mathcal I_j|}.$$
From Lemma~\ref{lem:ineq}, we obtain  that
$$\prod_{j=1}^i\left(1+\frac{1}{q^j}\right)^{|\mathcal I_j|}\le \prod_{j=1}^i\left(1+\frac{1}{q^j}\right)^{\frac{q^j}{j}}\le e^{\sum_{j=1}^i1/j}\le e\cdot i.$$
Therefore, we obtain that 
$$\sum_{i\ge 1}\sum_{f\in \mathcal I_i}\frac{1}{\ord(f)}\sum_{G\in \mathcal M(i, f)} \frac{1}{q^{\deg(G)}}\le e\sum_{i\ge 1}\sum_{f\in \mathcal I_i}\frac{i}{q^i\ord(f)}.$$
From items (iii) and (iv) of Lemma~\ref{lem:ord}, we have that

$$\sum_{i\ge 1}\sum_{f\in \mathcal I_i}\frac{i}{q^i\ord(f)}=\sum_{i\ge 1}\frac{i}{q^i}\cdot \sum_{E\ge 1\atop \ord_Eq=i}\frac{1}{E}\cdot \frac{\varphi(E)}{i}=\sum_{i\ge 1}\frac{1}{q^i}\sum_{E\ge 1\atop \ord_Eq=i}\frac{\varphi(E)}{E}.$$
Observe that $E|q^i-1$ whenever $\ord_Eq=i$. Moreover, $\frac{\varphi(E)}{E}<1$ for every $E\ge 1$. From these facts, we obtain that 
$$\sum_{E\atop \ord_Eq=i}\frac{\varphi(E)}{E}<\sigma_0(q^i-1).$$ From Lemma~\ref{lem:estimate-divisor}, $\sigma_0(q^i-1)=O(q^{i/2})$. Hence $$\sum_{i\ge 1}\frac{1}{q^i}\sum_{E\atop \ord_Eq=i}\frac{\varphi(E)}{E}=O\left(\sum_{i\ge 1}\frac{1}{q^{i/2}}\right).$$ Since $q^{1/2}>1$, the series $\sum_{i\ge 1}\frac{1}{q^{i/2}}$ converges, from where the result follows.

\item Fix $t\ge 1$. Similarly to the previous item, we have that 

$$R_F^*(t)=\sum_{G\in \mathcal M^*\atop 1\le \ord(G)\le t}\frac{1}{q^{\deg(G)}}\le 1+\sum_{1\le i\le t}\sum_{f\in I_i\atop 1\le \ord(f)\le t}\frac{e i}{q^i}=1+e\sum_{1\le i\le t}\frac{i}{q^i}\sum_{1\le E\le t\atop \ord_Eq=i}\frac{\varphi(E)}{i}.$$
Therefore, 

$$R_F^*(t)\le 1+e\sum_{1\le i\le t}\frac{\delta(i, t)}{q^i},$$
where $\delta(i, t)=\sum_{1\le E\le t\atop \ord_Eq=i}\varphi(E)$. If $q^i\ge t^2$, we have that 
$$\delta(i, t)= \sum_{1\le E\le t\atop \ord_Eq=i}\varphi(E)\le \sum_{1\le E\le t}E\le \frac{t(t+1)}{2}\le  t^2.$$
If $q^i<t^2$, we have the trivial bound $$\delta(i, t)\le \sum_{E|q^i-1}\varphi(E)=q^i-1<q^i.$$ In particular, 
\begin{align*}
    \sum_{1\le i\le t}\frac{\delta(i, t)}{q^i}<t^2\sum_{1\le i\le t\atop q^i\ge t^2}\frac{1}{q^i}+\sum_{1\le i\le t\atop q^i<t^2}1 &\le  t^2\sum_{i\ge 1\atop q^i\ge t^2}\frac{1}{q^i}+2\log_q t \\ {} &\le  \, t^2\cdot\frac{q}{t^2(q-1)}+2\log_q t=O(\log t).
\end{align*}

    \end{enumerate}
\end{proof}

\subsection{Proof of Corollary~\ref{cor:main}}

The case $k=0$ is trivial so we assume that $k\ge 1$. Set $F(X)=(X-1)^k$, hence $F(X)$ divides $x^{p^t\cdot u}-1$ for every integer $u\ge 1$, where $t=\lfloor\log_pk\rfloor+1$. Since $\frac{X^{p^tu}-1}{F(X)}=\left(\frac{X^u-1}{X-1}\right)^{p^t}(X-1)^{p^t-k}$ and $p^t>k$, we have that the irreducible factors of $\frac{X^{p^tu}-1}{F}$ and $X^u-1$ coincide. From Lemmas~\ref{lem:phi} and~\ref{eq:count}, we obtain that 
$$\lambda_{q, k}(p^tu)\ge \frac{\Phi_q\left(\frac{X^{p^tu}-1}{F}\right)}{q^n}=\frac{1}{q^k}\cdot \frac{\Phi_q(X^u-1)}{q^{\deg(X^u-1)}}=\frac{\lambda_{q, 0}(u)}{q^k}.$$
Therefore,
$$\frac{1}{y}\sum_{1\le n\le y}\lambda_{q, k}(n)\ge \frac{1}{y}\sum_{1\le p^tu\le y}\lambda_{q, k}(p^tu)\ge \frac{1}{yq^k}\sum_{1\le u\le y/p^t}\lambda_{q, 0}(u).$$
In particular,
$$\overline{\lambda}_{q, k}=\limsup\limits_{y\to+\infty} \frac{1}{y}\sum_{1\le n\le y}\lambda_{q, k}(n)\ge \limsup\limits_{y\to+\infty} \frac{1}{yq^k}\sum_{1\le u\le y/p^t}\lambda_{q, 0}(u)=\frac{\overline{\lambda}_{q, 0}}{p^tq^k}.$$
The inequalities $\overline{\lambda}_{q, 0}>0$ and $\overline{\lambda}_{q, 0}>1-\frac{1}{\sqrt{q}}-\frac{1}{q}$ for $q\ge 4$ follow by Theorems 4.2 and 4.8 of~\cite{R}, respectively.

\section*{Acknowledgments}
The author was supported by CNPq (309844/2021-5).

%\appendix

%\section{Appendices}

%Use only when absolutely necessary. They
%should come before the References. If there is more than one
%appendix, number them alphabetically. Number displayed equations
%occurring in the Appendix as (\ref{eq2}).
%\begin{equation}
%\mu(n, t) = {\sum^\infty_{i=1} 1(d_i < t, N(d_i)
%= n)}{\int^t_{\sigma=0} 1(N(\sigma) = n)d\sigma}\,.
%\label{eq2}
%\end{equation}

\end{document}